\documentclass[11pt]{amsart}
\usepackage{amsfonts,amssymb,amscd,amsmath,enumerate,verbatim,calc}

\def\NZQ{\mathbb}               

\def\QQ{{\NZQ Q}}
\def\ZZ{{\NZQ Z}}
\def\RR{{\NZQ R}}

\def\frk{\mathfrak}               

\def\Phi{{\frk N}}
\def\opn#1#2{\def#1{\operatorname{#2}}} 
\opn\chara{char} \opn\length{\ell} \opn\pd{pd} \opn\rk{rk}
\opn\projdim{proj\,dim} \opn\injdim{inj\,dim} \opn\rank{rank}
\opn\depth{depth} \opn\grade{grade} \opn\height{height}
\opn\embdim{emb\,dim} \opn\codim{codim}

\opn\Tr{Tr} \opn\bigrank{big\,rank}
\opn\superheight{superheight}\opn\lcm{lcm}
\opn\trdeg{tr\,deg}
\opn\reg{reg} \opn\lreg{lreg} \opn\ini{in} \opn\lpd{lpd}
\opn\size{size}\opn{\mult}{mult}
\opn\Div{Div} \opn\cl{cl} \opn\Cl{Cl}
\opn\Spec{Spec} \opn\Supp{Supp} \opn\supp{supp} \opn\Sing{Sing}
\opn\Ass{Ass} \opn\Min{Min}
\opn\Ann{Ann} \opn\Rad{Rad} \opn\Soc{Soc}
\opn\Syz{Syz} \opn\Im{Im} \opn\Ker{Ker} \opn\Coker{Coker}
\opn\Am{Am} \opn\Hom{Hom} \opn\Tor{Tor} \opn\Ext{Ext}
\opn\End{End} \opn\Aut{Aut} \opn\id{id}

\opn\nat{nat}
\opn\pff{pf}
\opn\Pf{Pf} \opn\GL{GL} \opn\SL{SL} \opn\mod{mod} \opn\ord{ord}
\opn\Gin{Gin}
\opn\Hilb{Hilb}\opn\adeg{adeg}\opn\std{std}\opn\ip{infpt}
\opn\Pol{Pol}
\opn\sat{sat}
\opn\Var{Var}
\opn\Gen{Gen}
\opn\vol{vol}
\opn\bx{Box}

\opn\aff{aff} \opn\con{conv} \opn\relint{relint} \opn\st{st}
\opn\lk{lk} \opn\cn{cn} \opn\core{core} \opn\vol{vol}
\opn\link{link} \opn\star{star}
\opn\gr{gr}

\def\Fc{{\mathcal F}}
\def\Pc{{\mathcal P}}

\def\pot#1#2{#1[\kern-0.28ex[#2]\kern-0.28ex]}

\opn\dirlim{\underrightarrow{\lim}}
\opn\inivlim{\underleftarrow{\lim}}
\def\Implies{\ifmmode\Longrightarrow \else
        \unskip${}\Longrightarrow{}$\ignorespaces\fi}
\def\implies{\ifmmode\Rightarrow \else
        \unskip${}\Rightarrow{}$\ignorespaces\fi}
\def\iff{\ifmmode\Longleftrightarrow \else
        \unskip${}\Longleftrightarrow{}$\ignorespaces\fi}

\let\:=\colon
\newtheorem{Theorem}{Theorem}[section]
\newtheorem{Lemma}[Theorem]{Lemma}

\newtheorem{Proposition}[Theorem]{Proposition}
\newtheorem{Remark}[Theorem]{Remark}

\newtheorem{Example}[Theorem]{Example}

\let\epsilon\varepsilon
\let\phi=\varphi
\let\kappa=\varkappa
\opn\dis{dis}
\def\pnt{{\raise0.5mm\hbox{\large\bf.}}}

\opn\Lex{Lex}

\begin{document}

\title{
Ehrhart polynomials of integral simplices \\
with prime volumes
}

\author{Akihiro Higashitani}
\thanks{
\\ \;\;\;\; 
The author is supported by JSPS Research Fellowship for Young Scientists. 
}
\address{Akihiro Higashitani, 
Department of Pure and Applied Mathematics,
Graduate School of Information Science and Technology,
Osaka University,
Toyonaka, Osaka 560-0043, Japan}
\email{a-higashitani@cr.math.sci.osaka-u.ac.jp}

\begin{abstract}
For an integral convex polytope $\Pc \subset \RR^N$ of dimension $d$, we call 
$\delta(\Pc)=(\delta_0, \delta_1, \ldots, \delta_d)$ the $\delta$-vector of $\Pc$ 
and $\vol(\Pc)=\sum_{i=0}^d\delta_i$ its normalized volume. In this paper, 
we will establish the new equalities and inequalities on $\delta$-vectors 
for integral simplices whose normalized volumes are prime. 
Moreover, by using those, we will classify all the possible $\delta$-vectors 
of integral simplices with normalized volume 5 and 7. 
\end{abstract}

\maketitle

\section*{Introduction}

One of the most fascinating problems on enumerative combinatorics 
is to characterize the $\delta$-vectors of integral convex polytopes. 

Let $\Pc \subset \RR^N$ be an {\em integral} convex polytope of dimension $d$, 
which is a convex polytope any of whose vertices has integer coordinates. 
Let $\partial \Pc$ denote the boundary of $\Pc$. Given a positive integer $n$, we define 
$$i(\Pc,n) = |n\Pc \cap \ZZ^N|, \;\;\;\;\; 
i^*(\Pc,n)=|n (\Pc \setminus \partial \Pc) \cap \ZZ^N|,$$ 
where $n\Pc = \{ n\alpha : \alpha \in \Pc \}$ 
and $|X|$ is the cardinality of a finite set $X$. 
The enumerative function $i(\Pc,n)$ has the following fundamental properties, 
which were studied originally in the work of Ehrhart \cite{Ehrhart}: 
\begin{itemize}
\item $i(\Pc,n)$ is a polynomial in $n$ of degree $d$; 
\item $i(\Pc,0) = 1$; 
\item (loi de r\'eciprocit\'e) 
$i^*(\Pc,n)=( - 1 )^d i(\Pc, - n)$ for every integer $n > 0$. 
\end{itemize}
This polynomial $i(\Pc,n)$ is called the {\em Ehrhart polynomial} of $\Pc$. 
We refer the reader to \cite[Chapter 3]{BeckRobins}, \cite[Part II]{HibiRedBook} 
or \cite[pp. 235--241]{StanleyEC} 
for the introduction to the theory of Ehrhart polynomials. 

We define the sequence 
$\delta_0, \delta_1, \delta_2, \ldots $ of integers by the formula 
\begin{eqnarray}\label{delta}
(1 - \lambda)^{d + 1}  \sum_{n=0}^{\infty} i(\Pc,n) \lambda^n 
= \sum_{i=0}^{\infty} \delta_i \lambda^i. 
\end{eqnarray}
Then, from a fundamental result on generating functions (\cite[Corollary 4.3.1]{StanleyEC}), 
we know that $\delta_i = 0$ for $i > d$. We call the integer sequence 
$$\delta(\Pc)= (\delta_0, \delta_1, \ldots, \delta_d),$$ 
which appears in \eqref{delta}, the {\em $\delta$-vector} of $\Pc$. 

The $\delta$-vector has the following properties: 
\begin{itemize}
\item $\delta_0=1$, $\delta_1 = |\Pc \cap \ZZ^N| - (d + 1)$ and 
$\delta_d = |(\Pc \setminus \partial \Pc) \cap \ZZ^N|.$ 
Hence, $\delta_1 \geq \delta_d$. In particular, when $\delta_1=\delta_d$, 
$\Pc$ must be a simplex. 
\item Each $\delta_i$ is nonnegative (\cite{StanleyDRCP}). 
\item If $(\Pc \setminus \partial \Pc) \cap \ZZ^N$ is nonempty, 
then one has $\delta_1 \leq \delta_i$ for every $1 \leq i \leq d - 1$ (\cite{HibiLBT}). 
\item The leading coefficient $(\sum_{i=0}^d\delta_i)/d!$ of $i(\Pc,n)$ 
is equal to the usual volume of $\Pc$ (\cite[Proposition 4.6.30]{StanleyEC}). 
In particular, the positive integer $\vol(\Pc) = \sum_{i=0}^d\delta_i$ 
is said to be the {\em normalized volume} of $\Pc$. 
\end{itemize}

Recently, the $\delta$-vectors of integral convex polytopes 
have been studied intensively. 
For example, see \cite{HigashiSS, Staple1, Staple2}.


There are two well-known inequalities on $\delta$-vectors. 
Let $s = \max\{ i : \delta_i \neq 0 \}$. One is 
\begin{eqnarray}\label{Stanley}
\delta_0 + \delta_1 + \cdots + \delta_i 
\leq \delta_s + \delta_{s-1} + \cdots + \delta_{s-i}, 
\;\;\;\;\; 0 \leq i \leq \lfloor s/2 \rfloor, 
\end{eqnarray}
which is proved by Stanley \cite{StanleyJPAA}, and another one is 
\begin{eqnarray}\label{Hibi}
\delta_d + \delta_{d-1} + \cdots + \delta_{d-i} 
\leq \delta_1 + \delta_2 + \cdots + \delta_i + \delta_{i+1}, 
\;\;\;\;\; 0 \leq i \leq \lfloor (d-1)/2 \rfloor, 
\end{eqnarray}
which appears in the work of Hibi \cite[Remark (1.4)]{HibiLBT}. 

On the classification problem on $\delta$-vectors of integral convex polytopes, 
the above inequalities \eqref{Stanley} and \eqref{Hibi} 
characterize the possible $\delta$-vectors completely 
when $\sum_{i=0}^d\delta_i \leq 3$ (\cite[Theorem 0.1]{HHN}). 
Moreover, when $\sum_{i=0}^d\delta_i = 4$, 
the possible $\delta$-vectors are determined completely 
by \eqref{Stanley} and \eqref{Hibi} together with an additional condition (\cite[Theorem 5.1]{HHNan}). 
Furthermore, by the proofs of \cite[Theorem 0.1]{HHN} and \cite[Theorem 5.1]{HHNan}, 
we know that all the possible $\delta$-vectors can be realized 
as the $\delta$-vectors of integral simplices when $\sum_{i=0}^d \delta_i \leq 4$. 
However, unfortunately, this is no longer true when $\sum_{i=0}^d \delta_i = 5$. 
(See \cite[Remark 5.2]{HHNan}.) 
Hence, for the further classifications of $\delta$-vectors with $\sum_{i=0}^d\delta_i \geq 5$, 
it is natural to study $\delta$-vectors of integral simplices at first. 
Even for non-simplex cases, 
since every convex polytope can be triangulated into finitely many simplices 
and we can compute the $\delta$-vecotor of an integral convex polytope from its triangulation, 
investigating $\delta$-vectors of integral simplices is an essential and important work.

In this paper, in particular, we establish some new constraints 
on $\delta$-vectors for integral simplices whose normalized volumes are prime numbers. 
The following theorem is our main result of this paper. 

\begin{Theorem}\label{main}
Let $\Pc$ be an integral simplex of dimension $d$ and 
$\delta(\Pc)=(\delta_0,\delta_1,\ldots,\delta_d)$ its $\delta$-vector. 
Suppose that $\vol(\Pc)=\sum_{i=0}^d \delta_i=p$ is an odd prime number. 
Let $i_1,\ldots,i_{p-1}$ be the positive integers such that 
$\sum_{i=0}^d\delta_it^i=1+t^{i_1}+\cdots+t^{i_{p-1}}$ 
with $1 \leq i_1 \leq \cdots \leq i_{p-1} \leq d$. Then, 
\begin{itemize}
\item[(a)] {\em (cf. \cite[Theorem 2.3]{HigashiSS})} one has 
\begin{eqnarray*}
i_1+i_{p-1}=i_2+i_{p-2}=\cdots=i_{(p-1)/2}+i_{(p+1)/2} \leq d+1; 
\end{eqnarray*}
\item[(b)] one has 
\begin{eqnarray*}
i_k+i_{\ell} \geq i_{k+ \ell} \; \text{ for } \; 1 \leq k \leq \ell \leq p-1 
\; \text{ with } \; k+ \ell \leq p-1. 
\end{eqnarray*}
\end{itemize}
\end{Theorem}

We give a proof of Theorem \ref{main} in Section \ref{sec:review}.

Now, we remark that the part (a) of Theorem \ref{main} is not a new result in some sense. 
In \cite[Theorem 2.3]{HigashiSS}, the author proved that 
for an integral simplex $\Pc$ with prime normalized volume, 
if $i_1+i_{p-1}=d+1$, then $\Pc$ is shifted symmetric, i.e., 
we have $i_1+i_{p-1}=i_2+i_{p-2}=\cdots=i_{(p-1)/2}+i_{(p+1)/2}$. 
Moreover, since every integral simplex with prime normalized volume is 
either a simplex with $i_1+i_{p-1}=d+1$ or a pyramid at height 1 over such simplex 
and taking such a pyramid does not change the normalized volume and 
the polynomial $1+\sum_{j=1}^{p-1}t^{i_j}$, 
we also obtain the equalities $i_1+i_{p-1}=i_2+i_{p-2}=\cdots=i_{(p-1)/2}+i_{(p+1)/2}$ 
on the case where $i_1+i_{p-1}<d+1$. 
On the other hand, in this paper, we give an another proof of this statement. 
More precisely, we give an elementary proof of Theorem \ref{main} (a) 
in terms of some abelian groups associated with integral simplices.

In addition, as an application of Theorem \ref{main}, 
we give a complete characterization of the possible $\delta$-vectors of 
integral simplices when $\sum_{i=0}^d\delta_i=5$ and 7. 

\begin{Theorem}\label{delta5}
Given a finite sequence $(\delta_0,\delta_1,\ldots,\delta_d)$ of nonnegative integers, 
where $\delta_0=1$ and $\sum_{i=0}^d\delta_i=5$, 
there exists an integral simplex $\Pc \subset \RR^d$ of dimension $d$ 
whose $\delta$-vector coincides with $(\delta_0,\delta_1,\ldots,\delta_d)$ 
if and only if $i_1,\ldots,i_4$ satisfy 
$i_1+i_4=i_2+i_3 \leq d+1$ and $i_k+i_{\ell} \geq i_{k+\ell}$ 
for $1 \leq k \leq \ell \leq 4$ with $k+\ell \leq 4$, 
where $i_1,\ldots,i_4$ are the positive integers such that 
$\sum_{i=0}^d\delta_it^i=1+t^{i_1}+\cdots+t^{i_4}$ 
with $1 \leq i_1 \leq \cdots \leq i_4 \leq d$. 
\end{Theorem}

\begin{Theorem}\label{delta7}
Given a finite sequence $(\delta_0,\delta_1,\ldots,\delta_d)$ of nonnegative integers, 
where $\delta_0=1$ and $\sum_{i=0}^d\delta_i=7$, 
there exists an integral simplex $\Pc \subset \RR^d$ of dimension $d$ 
whose $\delta$-vector coincides with $(\delta_0,\delta_1,\ldots,\delta_d)$ 
if and only if $i_1,\ldots,i_6$ satisfy $i_1+i_6=i_2+i_5=i_3+i_4 \leq d+1$ and 
$i_k+i_{\ell} \geq i_{k+\ell}$ for $1 \leq k \leq \ell \leq 6$ with $k+\ell \leq 6$, 
where $i_1,\ldots,i_6$ are the positive integers such that 
$\sum_{i=0}^d\delta_it^i=1+t^{i_1}+\cdots+t^{i_6}$ 
with $1 \leq i_1 \leq \cdots \leq i_6 \leq d$. 
\end{Theorem}

By virtue of Theorem \ref{main}, the ``Only if'' parts of Theorem \ref{delta5} and \ref{delta7} 
are obvious. A proof of the ``If'' part of Theomre \ref{delta5} 
is given in Section \ref{sec:classify5} and that of Theorem \ref{delta7} 
is given in Section \ref{sec:classify7}. 
Moreover, in Section \ref{sec:towards}, we note some problems 
towards the classification of Ehrhart polynomials of integral convex polytopes 
with general normalized volumes.

\bigskip

\section{A proof of Theorem \ref{main}}\label{sec:review}

The goal of this section is to give a proof of Theorem \ref{main}. 

First of all, we recall 
the well-known combinatorial technique how to compute 
the $\delta$-vector of an integral simplex. 
Given an integral simplex $\Fc \subset \RR ^N$ of dimension $d$ 
with the vertices $v_0, v_1, \ldots, v_d \in \RR^N$, we set 
\begin{eqnarray*}
\bx(\Fc)=\left\{ \alpha \in \ZZ^{N+1} : \alpha = \sum_{i=0}^d r_i(v_i,1), \;\;
0 \leq r_i < 1 \right\}. 
\end{eqnarray*}
We define the degree of $\alpha=\sum_{i=0}^{d}r_i(v_i,1) \in \bx(\Fc)$ 
with $\deg(\alpha)=\sum_{i=0}^d r_i$, i.e., the last coordinate of $\alpha$. 
Then we have the following 
\begin{Lemma}\label{compute}
Let $\delta(\Fc)=(\delta_0,\delta_1,\ldots,\delta_d).$ 
Then each $\delta_i$ is equal to the number of integer points $\alpha \in \bx(\Fc)$ with $\deg(\alpha)=i$. 
\end{Lemma}

Notice that $\bx(\Fc)$ has a structure of an abelian group with a unit ${\bf 0} \in \bx(\Fc)$, 
where ${\bf 0}=(0,\ldots,0)$. 
For $\alpha$ and $\beta$ in $\bx(\Fc)$ with 
$\alpha=\sum_{i=0}^dr_i(v_i,1)$ and $\beta=\sum_{i=0}^ds_i(v_i,1)$, 
where $r_i,s_i \in \QQ$ with $0 \leq r_i,s_i <1$, we define the operation in $\bx(\Fc)$ 
by setting $\alpha \oplus \beta := \sum_{i=0}^d\{r_i+s_i\}(v_i,1)$, 
where $\{r\}=r-\lfloor r \rfloor$ denotes the fractional part of a rational number $r$. 
(Throughout this paper, in order to distinguish the operation in $\bx(\Fc)$ 
from the usual addition, we use the notation $\oplus$, which is not a direct sum.) 


\smallskip

We prove Theorem \ref{main} by using the above notations. 

\begin{proof}[Proof of Theorem \ref{main}]
Let $v_0,v_1,\ldots,v_d$ be the vertices of the integral simplex $\Pc$ 
and $\bx(\Pc)$ the abelian group as above. 
Then, since $\vol(\Pc)=p$ is prime, it follows from Lemma \ref{compute} that 
$|\bx(\Pc)|$ is also prime. In particular, $\bx(\Pc) \cong \ZZ/p \ZZ$. 

(a) Write $g_1,\ldots,g_{p-1}$ for $(p-1)$ distinct elements 
belonging to $\bx(\Pc) \setminus \{{\bf 0}\}$ with $\deg(g_j)=i_j$ for $1 \leq j \leq p-1$, 
that is, $\bx(\Pc)=\{{\bf 0},g_1,\ldots,g_{p-1}\}$. 
Then, for each $g_j$, there exists its inverse $-g_j$ in $\bx(\Pc) \setminus \{{\bf 0}\}$. 
Let $-g_j=g_j'$. If $g_j$ has an expression $g_j=\sum_{q=0}^dr_q(v_q,1)$, 
where $r_q \in \QQ$ with $0 \leq r_q <1$, then its inverse has an expression 
$g_j'=\sum_{q=0}^d\{1-r_q\}(v_q,1)$. Thus, one has 
$$\deg(g_j)+\deg(g_j') = \sum_{q=0}^d(r_q+\{1-r_q\}) 
\leq \sum_{q=0}^d(r_q+1-r_q)=d+1$$ for all $1 \leq j \leq p-1$. 

For $1 \leq j_1 \not= j_2 \leq p-1$, let 
$g_{j_1}=\sum_{q=0}^dr_q^{(1)}(v_q,1)$ and $g_{j_2}=\sum_{q=0}^dr_q^{(2)}(v_q,1).$ 
Since $\bx(\Pc) \cong \ZZ/p \ZZ$, $g_{j_1}$ generates $\bx(\Pc)$. 
This implies that we can write $g_{j_2}$ and $g_{j_2}'$ as follows: 
$$g_{j_2}=\underbrace{g_{j_1} \oplus \cdots \oplus g_{j_1}}_t, \;\;\; 
g_{j_2}'=\underbrace{g_{j_1}' \oplus \cdots \oplus g_{j_1}'}_t$$ 
for some integer $t \in \{2,\ldots,p-1\}$. Thus, we have 
\begin{eqnarray*}
&&\sum_{q=0}^d(r_q^{(2)}+\{1-r_q^{(2)}\})=\deg(g_{j_2})+\deg(g_{j_2}') \\
&&\quad\quad =\deg(\underbrace{g_{j_1} \oplus \cdots \oplus g_{j_1}}_t) + 
\deg(\underbrace{g_{j_1}' \oplus \cdots \oplus g_{j_1}'}_t)= 
\sum_{q=0}^d(\{tr_q^{(1)}\}+\{t(1-r_q^{(1)})\}). 
\end{eqnarray*}
Moreover, since 
$\underbrace{g_{j_1} \oplus \cdots \oplus g_{j_1}}_p={\bf 0}$, 
we have $\{pr_q^{(1)}\}=0$ for $0 \leq q \leq d$. 
This means that 
the denominator of each rational number $r_q^{(1)}$ must be $p$. 
Hence, if $0< r_q^{(1)} <1$ (resp. $0< \{1-r_q^{(1)}\} <1$), 
then $0<\{tr_q^{(1)}\}<1$ (resp. $0<\{t(1-r_q^{(1)})\}<1$), 
so $r_q^{(1)}+\{1-r_q^{(1)}\}=\{tr_q^{(1)}\}+\{t(1-r_q^{(1)})\}=1$. 
In addition, obviously, if $r_q^{(1)}=\{1-r_q^{(1)}\}=0$, 
then $\{tr_q^{(1)}\}=\{t(1-r_q^{(1)})\}=0$, 
so $r_q^{(1)}+\{1-r_q^{(1)}\}=\{tr_q^{(1)}\}+\{t(1-r_q^{(1)})\}=0$. 
Thus, $\deg(g_{j_1})+\deg(g_{j_1}')=\deg(g_{j_2})+\deg(g_{j_2}')$. 
Let $\deg(g_j')=i_j'$. Then we obtain 
$$i_1+i_1'=\cdots=i_{(p-1)/2}+i_{(p-1)/2}' 
(=i_{(p+1)/2}+i_{(p+1)/2}'=\cdots=i_{p-1}+i_{p-1}') \leq d+1.$$ 
Our work is to show that $i_j'=i_{p-j}$ for $1 \leq j \leq (p-1)/2$. 

First, we consider $i_1'$. Suppose that $i_1'\not=i_{p-1}$. 
Then, there is $m \in \{1,\ldots,p-2\}$ with $i_1'=i_m<i_{p-1}$. 
Thus, it follows that 
$$i_{p-1}+i_{p-1}'=i_1+i_1'=i_1+i_m <i_1+i_{p-1} \leq i_{p-1}' +i_{p-1},$$ 
a contradiction. Thus, $i_1'$ must be $i_{p-1}$. 
Next, we consider $i_2'$. 
Since $g_{i_2}' \not=g_{i_1}$ and $g_{i_2}' \not=g_{i_{p-1}},$ 
we may consider $i_2'$ among $\{i_2,\ldots,i_{p-2}\}$. 
Then, the same discussion can be done. Hence, $i_2'=i_{p-2}$. 
Similarly, we have $i_3'=i_{p-3}, \ldots, i_{(p-1)/2}'=i_{(p+1)/2}$. 

Therefore, we obtain the desired 
$$i_1+i_{p-1}=i_2+i_{p-2}=\cdots=i_{(p-1)/2}+i_{(p+1)/2} \leq d+1.$$ 

(b) Let $k$ and $\ell$ be integers with $1 \leq k \leq \ell \leq p-1$ and $k+\ell \leq p-1$. 
Write $g_1,\ldots,g_{\ell} \in \bx(\Pc) \setminus \{{\bf 0}\}$ 
for $\ell$ distinct elements with $\deg(g_j)=i_j$ for $1 \leq j \leq \ell$ 
and set $A=\{g_1,\ldots,g_{\ell}\} \cup \{{\bf 0}\}$ and 
$B=\{g_1,\ldots,g_k\} \cup \{{\bf 0}\}$. 
Now, Cauchy--Davenport Theorem (cf \cite[Theorem 2.2]{Nath}) 
guarantees that $|A \oplus B| \geq \min\{ p, |A|+|B|-1\}$, 
where $A \oplus B=\{ a \oplus b: a \in A, b \in B\}$. 
Clearly, {\bf 0} belongs to $A \oplus B$. Moreover, since $|A|+|B|-1=k+\ell+1 \leq p$, 
it follows that $A \oplus B$ contains at least $(k+ \ell)$ distinct elements in 
$\bx(\Pc) \setminus \{{\bf 0}\}$. In addition, for each $g \in A \oplus B$, 
$g$ satisfies $\deg(g) \leq i_k+i_{\ell}$. 
In fact, for non-zero elements $g_j \in A$ and $g_{j'} \in B$, 
if they have expressions 
$g_j=\sum_{q=0}^dr_q(v_q,1)$ and $g_{j'}=\sum_{q=0}^dr_q'(v_q,1),$ 
then one has $$ \deg(g_j \oplus g_{j'}) = \sum_{q=0}^d\{r_q+r_q'\} 
\leq \sum_{q=0}^d(r_q+r_q') = i_j+i_{j'} \leq i_k+i_{\ell}.$$ 

Hence, from the definition of $i_1,\ldots,i_{p-1}$, 
we obtain the inequalities $i_k+i_{\ell} \geq i_{k+\ell}$ 
for $1 \leq k \leq \ell \leq p-1$ with $k+\ell \leq p-1$, as desired. 
\end{proof}

It is easy to see that we can reduce the inequalities in Theorem \ref{main} (b) 
by using the equalities $i_1+i_{p-1}=i_2+i_{p-2}=\cdots=i_{(p-1)/2}+i_{(p+1)/2}$ as follows: 
\begin{eqnarray}\label{reduce}
i_k+i_{\ell} \geq i_{k+\ell} \;\text{ for }\; 
1 \leq k \leq \left\lfloor \frac{p-1}{3} \right\rfloor 
\; \text{ and } \; k \leq \ell \leq \left\lfloor \frac{p-k}{2} \right\rfloor. \end{eqnarray} 
In fact, when $k > \lfloor (p-1)/3 \rfloor$, since $p$ is prime, we have $k > p/3$. 
Thus, $k + 2 \ell \geq 3k > p$. By using $i_{k+\ell}+i_{p-k-\ell}=i_{\ell}+i_{p-\ell}$, 
we obtain $i_k+i_{\ell}-i_{k+\ell}=i_k+i_{p-k-\ell} - i_{p-\ell} \geq 0$, 
which is $i_k+i_{p-k-\ell} \geq i_{p-\ell}$, where $p-k-\ell < \ell$. 
Similarly, when $\ell > \lfloor (p-k)/2 \rfloor$, we have $k+2\ell > p$. 
Thus, we can deduce $i_k+i_{p-k-\ell} \geq i_{p-\ell}$. 

Moreover, some of the inequalities described in Theorem \ref{main} 
follow from \eqref{Stanley} and \eqref{Hibi}. 
\begin{Proposition}\label{equiv}
Let $\Pc$ be an integral convex polytope of dimension $d$ 
with its $\delta$-vector $(\delta_0,\delta_1,\ldots,\delta_d)$ 
and $i_1,\ldots,i_{m-1}$ the positive integers such that 
$\sum_{i=0}^d\delta_it^i=1+t^{i_1}+\cdots+t^{i_{m-1}}$ 
with $1 \leq i_1 \leq \cdots \leq i_{m-1} \leq d$, where $m=\sum_{i=0}^d\delta_i$. \\
{\em (a)} The inequalities $i_j+i_{m-j-1} \geq i_{m-1}$, where $1 \leq j \leq m-2$, 
are equivalent to \eqref{Stanley}. \\
{\em (b)} The inequalities $i_j+i_{m-j} \leq d+1$, where $1 \leq j \leq m-1$, 
are equivalent to \eqref{Hibi}. 
\end{Proposition}
\begin{proof}
(a) 
For each $1 \leq j \leq m-2$, the inequality 
$\delta_0+\cdots+\delta_{i_j} \leq \delta_s+\cdots+\delta_{s-i_j}$ 
follows from \eqref{Stanley}. 
Then its left-hand side is at least $j+1$ by the definition of $i_j$. 
Thus, in particular, its right-hand side is at least $j+1$. 
On the other hand, since $s=i_{m-1}$, it must be $i_{m-1} - i_j \leq i_{m-j-1}$, 
which means $i_j+i_{m-j-1} \geq i_{m-1}$. 
On the contrary, assume that $i_j + i_{m-j-1} \geq i_{m-1}$. 
For each $k$ with $1 \leq k < i_{m-1}=s$, 
there exists a unique $j$ with $0 \leq j \leq m-1$ such that $i_j \leq k < i_{j+1}$, 
where $i_0=0$ and $i_m=d+1$. Thus, 
\begin{align*}
&\delta_s+\cdots+\delta_{s-k} - (\delta_0+\cdots+\delta_k) = 
\delta_{i_{m-1}}+\cdots+\delta_{i_{m-1}-k}-(j+1) \\
&\quad\quad\quad
\geq \delta_{i_{m-1}}+\cdots+\delta_{i_{m-1}-i_j}-(j+1) 
\geq \delta_{i_{m-1}}+\cdots+\delta_{i_{m-j-1}} - (j+1) \geq 0. 
\end{align*}
(b) 
For each $1 \leq j \leq m-1$, the inequality 
$\delta_1+\cdots+\delta_{d+1-i_j} \geq \delta_d + \cdots+ \delta_{i_j}$ 
follows from \eqref{Hibi}. 
Then its right-hand side is at least $m-j$. 
Thus, it must be $d+1-i_j \geq i_{m-j}$, which means $i_j+i_{m-j} \leq d+1$. 
On the contrary, assume that $i_j+i_{m-j} \leq d+1$. 
For each $k$ with $1 \leq k <d$, 
there exists a unique $j$ with $1 \leq j \leq m$ such that 
$i_{j-1} < k \leq i_j$. Thus, 
\begin{align*}
&\delta_1+\cdots+\delta_{d+1-k} - (\delta_d+\cdots+\delta_k) \geq 
\delta_1+\cdots+\delta_{d+1-i_j} - (m-j) \\
&\quad\quad\quad\quad\quad\quad\quad \geq 
\delta_1+\cdots+\delta_{i_{m-j}} - (m-j) \geq 0. 
\end{align*}
\end{proof}

As is shown above, the inequalities $i_j+i_{m-j-1} \geq i_{m-1}$ and $i_j+i_{m-j} \leq d+1$ 
are not new ones. Howover, the inequalities $i_k+i_{\ell} \geq i_{k+\ell}$ include 
many new ones. See Remark \ref{rem7} and Example \ref{77} below. 


\smallskip

\section{The possible $\delta$-vectors of integral simplices 
with $\sum_{i=0}^d \delta_i=5$}
\label{sec:classify5} 

In this section, we give a proof of the ``If'' part of Theorem \ref{delta5}, 
namely, we classify all the possible $\delta$-vectors of integral simplices 
with normalized volume 5. 

Let $(\delta_0,\delta_1,\ldots,\delta_d)$ be a nonnegative integer sequence 
with $\delta_0=1$ and $\sum_{i=0}^d\delta_i=5$ 
which satisfies $i_1+i_4=i_2+i_3 \leq d+1$, $2i_1 \geq i_2$ and $i_1+i_2 \geq i_3$, 
where $i_1,\ldots,i_4$ are the positive integers such that 
$\sum_{i=0}^d\delta_it^i=1+t^{i_1}+\cdots+t^{i_4}$ 
with $1 \leq i_1 \leq \cdots \leq i_4 \leq d$. 
By virtue of Theorem \ref{main}, these are necessary conditions 
for $(\delta_0,\delta_1,\ldots,\delta_d)$ to be a $\delta$-vector of some integral simplex. 
We notice that $i_1+i_3 \geq i_4$ (resp. $2i_2 \geq i_4$) is equivalent to 
$2i_1 \geq i_2$ (resp. $i_1+i_2 \geq i_3$) since $i_1+i_4=i_2+i_3$. 
From the conditions $\delta_0=1$, $\sum_{i=0}^d\delta_i=5$ and $i_1+i_4=i_2+i_3$, 
only the possible sequences look like 
\begin{itemize}
\item[(i)] $(1,0,\ldots,0,4,0,\ldots,0)$; 
\item[(ii)] $(1,0,\ldots,0,2,0,\ldots,0,2,0,\ldots,0)$; 
\item[(iii)] $(1,0,\ldots,0,1,0,\ldots,0,2,0,\ldots,0,1,0,\ldots,0)$; 
\item[(iv)] $(1,0,\ldots,0,1,0,\ldots,0,1,0,\ldots,0,1,0,\ldots,0,1,0,\ldots,0)$. 
\end{itemize}
Our work is to find integral simplices whose $\delta$-vectors are of the above forms. 

To construct integral simplices, we define the following $d \times d$ integer matrix: 
\begin{eqnarray}\label{M}
A_m(d_1,\ldots,d_{m-1})=\begin{pmatrix}
1 &       &     &  &  &       &  \\
  &\ddots &     &  &  &       &  \\
  &       &1    &  &  &       &  \\
* &\cdots &*    &m &  &       &  \\
  &       &     &  &1 &       &  \\
  &       &     &  &  &\ddots &  \\
  &       &     &  &  &       &1
\end{pmatrix},
\end{eqnarray}
where $m$ is a positive integer, 
there are $d_j$ $j$'s among the $*$'s for $j=1,\ldots,m-1$ 
and the rest entries are all 0. 
Clearly, it must be $d_j \geq 0$ and $d_1+\cdots+d_{m-1} \leq d-1$. 
Thus, by determining $d_1,\ldots,d_{m-1}$, we obtain an integer matrix $A_m(d_1,\ldots,d_{m-1})$ 
and we define the integral simplex $\Pc_m(d_1,\ldots,d_{m-1})$ of dimension $d$ 
from the matrix \eqref{M} as follows: 
$$\Pc_m(d_1,\ldots,d_{m-1})=\con(\{{\bf 0},v_1,\ldots,v_d\}) \subset \RR^d,$$ 
where $v_i$ is the $i$th row vector of \eqref{M}. 
The following lemma enables us to compute $\delta(\Pc_m(d_1,\ldots,d_{m-1}))$ easily. 

\begin{Lemma}[{\cite[Corollary 3.1]{HHNan}}]\label{lemmaM} 
Let $\delta(\Pc_m(d_1,\ldots,d_{m-1}))=(\delta_0,\delta_1,\ldots,\delta_d)$. Then 
$$\sum_{i=0}^d\delta_it^i=1+\sum_{i=1}^{m-1}t^{1-s_i},$$ where 
$$
s_i=\left\lfloor \frac{i}{m} - \sum_{j=1}^{m-1}\left\{ \frac{ij}{m} \right\}d_j \right\rfloor 
\;\;\;\;\; \text{for} \;\; i=1,\ldots,m-1. 
$$
\end{Lemma}

\bigskip

Let $m=5$. 
In the sequel, in each case of (i) -- (iv) above, 
by giving concrete values of $d_1,\ldots,d_4$, we obtain the matrix $A_5(d_1,\ldots,d_4)$ 
and hence the integral simplex $\Pc_5(d_1,\ldots,d_4)$ whose $\delta$-vector 
looks like each of (i) -- (iv). 
The $\delta$-vectors of such simplices can be computed by Lemma \ref{lemmaM}.

\subsection{The case (i)}

First, let us consider the case (i), namely, 
the nonnegative integer sequence like $(1,0,\ldots,0,4,0,\ldots,0)$, 
which means that $i_1=i_2=i_3=i_4$. 
Set $i_1=\cdots=i_4=i$. Then, of course, $i - 1 \geq 0$. 
Moreover, from our conditions, one has $2i=i_1+i_4=i_2+i_3 \leq d+1$, 
that is, $2i-2 \leq d-1$. 
Hence, we can define $\Pc_5(0,i-1,i-1,0)$ and calculate 
$$s_1=s_2=s_3=s_4=
\left\lfloor \frac{1}{5} - \sum_{j=1}^4\left\{ \frac{j}{5} \right\}d_j \right\rfloor= 
\left\lfloor \frac{1}{5} - \frac{2}{5}(i-1) 
- \frac{3}{5}(i-1)\right\rfloor=-i+1.$$ 
This implies that 
$\delta(\Pc_5(0,i-1,i-1,0))$ coincides with $(1,0,\ldots,4,0,\ldots,0)$ 
from Lemma \ref{lemmaM}, where $\delta_i=4$. 
Remark that $i$ should be at most $(d+1)/2$ by our condition (Theorem \ref{main} (a)). 

\smallskip

Similar discussions can be applied to the rest cases (ii) -- (iv).

\subsection{The case (ii)}

In this case, we have $i_1=i_2$ and $i_3=i_4$. 
Let $i_1=i_2=i$ and $i_3=i_4=j$. Thus, one has 
$2i \geq j$, $2j-2i-2 \geq 0$ and $i+j-2 \leq d-1$ from our conditions. 
Hence, we can define $\Pc_5(0,i,2i-j,2j-2i-2)$ 
and its $\delta$-vector coincides with (ii) 
since $s_1=s_2=-j+1$ and $s_3=s_4=-i+1$. 

\subsection{The case (iii)}

Let $i_1=i,i_2=i_3=j$ and $i_4=k$. Thus, one has 
$2i \geq j$, $3j-3i-2 \geq 0$ and $2j-2 \leq d-1$. 
Hence, we can define $\Pc_5(0,2i-j,i,3j-3i-2)$ 
and its $\delta$-vector coincides with (iii) 
since $s_1=-2j+i+1=-k+1$, $s_2=s_3=-j+1$ and $s_4=-i+1$. 

\subsection{The case (iv)}

In this case, one has $2i_1 \geq i_2, i_1+i_2 \geq i_3, i_2+2i_3-3i_1-2 \geq 0$ 
and $i_2+i_3-2 \leq d-1$. 
Hence, we can define $\Pc_5(0,2i_1-i_2,i_1+i_2-i_3,i_2+2i_3-3i_1-2)$ 
and its $\delta$-vector coincides with (iv) 
since $s_1=i_1-i_2-i_3+1=-i_4+1, s_2=-i_3+1, s_3=-i_2+1$ and $s_4=-i_1+1$.

\begin{Remark}{\em 
(a) The classification of the case (iv) 
is essentially given in \cite[Lemma 4.3]{HigashiSS}. \\
(b) Since $i_1+i_4=i_2+i_3$, the inequalities $2i_1 \geq i_2$ and $i_1+i_2 \geq i_3$ 
can be obtained from \eqref{Stanley} (see Proposition \ref{equiv}). 
Thus, the possible $\delta$-vectors of integral simplices 
with normalized volume 5 can be essentially characterized only by 
Theorem \ref{main} (a) and the inequalities \eqref{Stanley}. 
}\end{Remark}

\smallskip

\section{The possible $\delta$-vectors of integral simplices 
with $\sum_{i=0}^d \delta_i=7$}
\label{sec:classify7} 

In this section, similar to the previous section, 
we classify all the possible $\delta$-vectors of integral simplices 
with normalized volume 7. 

Let $(\delta_0,\delta_1,\ldots,\delta_d)$ be a nonnegative integer sequence 
with $\delta_0=1$ and $\sum_{i=0}^d\delta_i=7$ 
which satisfies $i_1+i_6=i_2+i_5=i_3+i_4 \leq d+1$, 
$i_1+i_l \geq i_{l+1}$ for $l=1,2,3$ and $2i_2 \geq i_4$, 
where $i_1,\ldots,i_6$ are the positive integers such that 
$\sum_{i=0}^d\delta_it^i=1+t^{i_1}+\cdots+t^{i_6}$ 
with $1 \leq i_1 \leq \cdots \leq i_6 \leq d$. 
Since $i_1+i_6=i_2+i_5=i_3+i_4$, we need not consider the inequalities 
$i_1+i_4 \geq i_5$, $i_1+i_5 \geq i_6$, 
$i_2+i_3 \geq i_5$, $i_2+i_4 \geq i_6$ and $2i_3 \geq i_6$. 
From the conditions $\delta_0=1$, $\sum_{i=0}^d\delta_i=7$ and $i_1+i_6=i_2+i_5=i_3+i_4$, 
only the possible sequences look like 
\begin{itemize}
\item[(i)] $(1,0,\ldots,0,6,0,\ldots,0)$; 
\item[(ii)] $(1,0,\ldots,0,3,0,\ldots,0,3,0,\ldots,0)$; 
\item[(iii)] $(1,0,\ldots,0,1,0,\ldots,0,4,0,\ldots,0,1,0,\ldots,0)$; 
\item[(iv)] $(1,0,\ldots,0,2,0,\ldots,0,2,0,\ldots,0,2,0,\ldots,0)$; 
\item[(v)] $(1,0,\ldots,0,1,0,\ldots,0,2,0,\ldots,0,2,0,\ldots,0,1,0,\ldots,0)$; 
\item[(vi)] $(1,0,\ldots,0,2,0,\ldots,0,1,0,\ldots,0,1,0,\ldots,0,2,0,\ldots,0)$; 
\item[(vii)] $(1,0,\ldots,0,1,0,\ldots,0,1,\ldots,0,2,0,\ldots,0,1,0,\ldots,0,1,0,\ldots,0)$; 
\item[(viii)] $(1,0,\ldots,0,1,0,\ldots,0,1,\ldots,0,1,0,\ldots,0,1,0,\ldots, 
0,1,0,\ldots,0,1,0,\ldots,0)$. 
\end{itemize}

%
%

\subsection{The case (i)} 
Let $i_1=\cdots=i_6=i$. Thus, one has $i-1 \geq 0$ and $2i-2 \leq d-1$ from our conditions. 
Hence, we can define $\Pc_7(0,0,i-1,i-1,0,0)$. By Lemma \ref{lemmaM}, 
$\delta(\Pc_7(0,0,i-1,i-1,0,0))$ coincides with (i) since $s_1=\cdots=s_6=-i+1$. 

\subsection{The case (ii)} 
Let $i_1=\cdots=i_3=i$ and $i_4=\cdots=i_6=j$. Thus, one has 
$j-i \geq 0, 2i \geq j, 2j-2i-2 \geq 0$ and $i+j-2 \leq d-1$. 
Hence, we can define $\Pc_7(0,j-i,2i-j,2i-j,0,2j-2i-2)$ 
and its $\delta$-vector coincides with (ii) 
since $s_1=s_2=s_3=-j+1$ and $s_4=s_5=s_6=-i+1$. 

\subsection{The case (iii)}
Let $i_1=i, i_2=\cdots=i_5=j$ and $i_6=k$. Thus, one has 
$i+j \geq k, k-j \geq 0, k-i-1 \geq 0, i-1 \geq 0$ and $i+k-2 \leq d-1$. 
Hence, we can define $\Pc_7(i+j-k,k-j,k-i-1,0,0,i-1)$ 
and its $\delta$-vector coincides with (iii) 
since $s_1=\frac{-4i+j-4k}{7}+1=-j+1, s_2=\frac{-i+2j-8k}{7}+1=-k+1, 
s_3=\frac{-5i+3j-5k}{7}+1=-j+1, s_4=\frac{-2i-3j-2k}{7}+1=-j+1, 
s_5=\frac{-6i-2j+k}{7}+1=-i+1$ and $s_6=\frac{-3i-j-3k}{7}+1=-j+1$. 

\subsection{The case (iv)}
Let $i_1=i_2=i, i_3=i_4=j$ and $i_5=i_6=k$. Thus, one has 
$i-1 \geq 0, i+j \geq k, 3k-3j-1 \geq 0$ and $2i-2j+2k-2=i+k-2 \leq d-1$. 
Hence, we can define $\Pc_7(0,0,i-1,i+j-k,0,3k-3j-1)$ 
and its $\delta$-vector coincides with (iv) 
since $s_1=s_2=-i+2j-2k+1=-k+1, s_3=s_4=-i+j-k+1=-j+1$ and $s_5=s_6=-i+1$.

\subsection{The case (v)}
Let $i_1=k_1, i_2=i_3=k_2, i_4=i_5=k_3$ and $i_6=k_4$. Thus, one has 
$2k_1 \geq k_2, k_2-k_1 \geq 0, k_1+k_2 \geq k_3, 2k_3-2k_1-2 \geq 0$ 
and $k_2+k_3-2 \leq d-1$. 
Hence, we can define $\Pc_7(0,2k_1-k_2,0,k_2-k_1,k_1+k_2-k_3,2k_3-2k_1-2)$ 
and its $\delta$-vector coincides with (v) 
since $s_1=k_1-k_2-k_3+1=-k_4+1, s_2=s_3=-k_3+1, s_4=s_5=-k_2+1$ and $s_6=-k_1+1$. 

\subsection{The case (vi)}
Let $i_1=i_2=k_1, i_3=k_2, i_4=k_3$ and $i_5=i_6=k_4$. Thus, one has 
$k_3-k_2-1 \geq 0, k_1+k_2 \geq k_3, 2k_1 \geq k_3, k_2+2k_3-3k_1-1 \geq 0$ 
and $k_2+k_3-2 \leq d-1$. 
Hence, we can define $\Pc_7(0,k_3-k_2-1,k_1+k_2-k_3,2k_1-k_3,0,k_2+2k_3-3k_1-1)$ 
and its $\delta$-vector coincides with (vi) 
since $s_1=s_2=k_1-k_2-k_3+1=-k_4+1, s_3=-k_3+1, s_4=-k_2+1$ and $s_5=s_6=-k_1+1$. 

\subsection{The case (vii)}
Let $i_1=k_1, i_2=k_2, i_3=i_4=k_3, i_5=k_4$ and $i_6=k_5$. Thus, one has 
$2k_1 \geq k_2, k_1+k_2 \geq k_3, k_2 - k_1 \geq 0, 3k_3-2k_1-k_2-2 \geq 0$ and 
$2k_3-2 \leq d-1$. 
Hence, we can define $\Pc_7(0,0,2k_1-k_2,k_1+k_2-k_3,k_2-k_1,3k_3-2k_1-k_2-2)$ 
and its $\delta$-vector coincides with (vii) 
since $s_1=k_1-2k_3+1=-k_5+1, s_2=k_2-2k_3+1=-k_4+1, 
s_3=s_4=-k_3+1, s_5=-k_2+1$ and $s_1=-k_1+1$. 

\subsection{The case (viii)}
In this case, one has $i_1+i_2 \geq i_3, 2i_2 \geq i_4, i_3+2i_4-2i_1-i_2-2 \geq 0, 
2i_1 \geq i_2, i_1+i_3 \geq i_4$ and $i_3+i_4-2 \leq d-1$. 
Hence, we can define 
$\Pc_7(0,i_1+i_2-i_3,i_1+i_3-2i_2,0,2i_2-i_4,i_3+2i_4-2i_1-i_2-2)$ 
if $i_1+i_3 \geq 2i_2$ and 
$\Pc_7(0,2i_1-i_2,0,2i_2-i_1-i_3,i_1+i_3-i_4,i_3+2i_4-2i_1-i_2-2)$ 
$i_1+i_3 \leq 2i_2$. 
Moreover, each of their $\delta$-vectors coincides with (viii) 
since $s_1=i_1-i_3-i_4+1=-i_6+1$, $s_2=i_2-i_3-i_4+1=-i_5+1$, 
$s_3=-i_4+1$, $s_4=-i_3+1$, $s_5=-i_2+1$ and $s_6=-i_1+1$. 

\begin{Remark}\label{rem7}{\em 
When we discuss the cases (vi) and (viii), we need the new inequality $2i_2 \geq i_4$. 
In fact, for example, the sequence $(1,0,2,0,1,1,0,2,0)$ 
cannot be the $\delta$-vector of an integral simplex, 
although this satisfies $i_1+i_l \geq i_{l+1}, l=1,2,3$. 
Similarly, the sequence $(1,0,1,1,0,1,0,1,0,1,1,0)$ is also impossible 
to be the $\delta$-vector of an integral simplex, 
although this satisfies $i_1+i_l \geq i_{l+1}, l=1,2,3$. 
}\end{Remark}

More generally, the following example shows that 
many inequalities $i_k+i_{\ell} \geq i_{k+\ell}$ are required to verify 
whether a given integer sequence is a $\delta$-vector of some integral simplex.

\begin{Example}\label{77}{\em 
For a prime number $p$ with $p \geq 7$, 
let $k$ and $\ell$ be positive integers satisfying the condition described in \eqref{reduce}. 
Let us consider the integer sequence 
$$(\delta_0,\delta_1,\ldots,\delta_d)=
(1,0,\ell,0,\underbrace{1,\ldots,1}_{p-2\ell-1},0,\ell,0) \in \ZZ^{d+1},$$ 
where $d=p-2\ell+5$. 
Then $i_1=\cdots=i_{\ell}=2$, $i_j=j-\ell+3$ for $j=\ell+1,\ldots,p-\ell-1$ 
and $i_{p-\ell}=\cdots=i_{p-1}=p-2\ell+4$, 
where $i_1,\ldots,i_{p-1}$ are the positive integers such that 
$\sum_{i=0}^d \delta_i t^i = 1+t^{i_1}+\cdots+t^{i_{p-1}}$ 
with $1 \leq i_1 \leq \cdots \leq i_{p-1} \leq d$. 
Thus, one has $i_k+i_{\ell}=4$ but $i_{k+\ell} = k+3$ or $i_{k+\ell}=p-2\ell+4$. 
In fact, since 
$$p-\ell - (k+\ell) \geq p-k-2\lfloor(p-k)/2\rfloor \geq p-k-p+k =0,$$ 
we have $\ell+1 \leq k+ \ell \leq p-\ell-1$ or $k+\ell=p-\ell$. 
Hence, this integer sequence satisfies none of the inequalities 
$i_k+i_{\ell} \geq i_{k+\ell}$ when $k$ and $\ell$ satisfy 
$k \geq 2$ and the condition in \eqref{reduce}. 
On the other hand, this satisfies both 
$i_j+i_{p-j}=d+1$ for $1 \leq j \leq p-1$ and 
$i_j+i_{p-j-1} \geq i_{p-1}$ for $1 \leq j \leq p-2$. 

Remark that since $\delta_1=0$, if there exists an integral convex polytope of dimension $d$ 
whose $\delta$-vector equals this sequence, then it must be a simplex. 
Therefore, thanks to Theorem \ref{main} (b), 
we can claim that there exists no integral convex polytope whose $\delta$-vector equals this sequence, 
while we cannot determine whether this integer sequence is a $\delta$-vector of 
some integral convex polytope only by \eqref{Stanley} and \eqref{Hibi}. 
}\end{Example}

\smallskip

\section{Towads the classification of Ehrhart polynomials 
with general normalized volumes}\label{sec:towards}

Finally, we note some future problems on the classification of Ehrhart polynomials 
of integral convex polytopes.

\subsection{Higher prime case}

Remark that we cannot characterize 
the possible $\delta$-vectors of integral simplices 
with higher prime normalized volumes only by Theorem \ref{main}, 
that is, Theorem \ref{main} is not sufficient. 
In fact, since the volume of an integral convex polytope 
containing a unique integer point in its interior has an upper bound, 
if $p$ is a sufficiently large prime number, then the integer sequence $(1,1,p-3,1)$ 
cannot be a $\delta$-vector of any integral simplex of dimension 3, 
although $(1,1,p-3,1)$ satisfies all the conditions of Theorem \ref{main}.

\subsection{Non-prime case}
We also remark that Theorem \ref{main} is not true when $\sum_{i=0}^d \delta_i$ is not prime in general. 
In fact, for example, there exists an integral simplex of dimension 5 
whose $\delta$-vector is $(1,1,0,2,0,0)$ (\cite[Theorem 5.1]{HHNan}), 
while this satisfies neither $i_1+i_3=i_2+i_2$ nor $2i_1 \geq i_2$, where $i_1=1$ and $i_2=i_3=3$.

More generally, for a non-prime number integer $m=gq$, 
where $g>1$ is the least prime divisor of $m$, 
let $d=m+1$ and $\Pc=\Pc_m(d_1,\ldots,d_{m-1})$ with $d_g=d-1$. 
Then, from Lemma \ref{lemmaM}, we have $\delta(\Pc)=(\delta_0,\delta_1,\ldots,\delta_d)$, where 
\begin{eqnarray*}
\delta_i=
\begin{cases}
1 &i=0, \\
g-1 &i=1, \\
g &i=g+1,2g+1,\ldots,(q-1)g+1. 
\end{cases}
\end{eqnarray*}
This $\delta$-vector satisfies neither $i_1+i_{gq-1}=i_g+i_{(q-1)g}$ nor $i_1+i_{g-1} \geq i_g$. 

On the other hand, Proposition \ref{equiv} is true even for non-prime normalized volume case 
and we also know other analogue of Theorem \ref{main} for such case as follows. 
\begin{Proposition}
Let $\Pc$ be an integral simplex of dimension $d$ with its $\delta$-vector 
$\delta(\Pc)=(\delta_0,\delta_1,\ldots,\delta_d)$ and 
$i_1,\ldots,i_{m-1}$ the positive integers such that 
$\sum_{i=0}^d\delta_it^i=1+t^{i_1}+\cdots+t^{i_{m-1}}$ with $1 \leq i_1 \leq \cdots \leq i_{m-1} \leq d$, 
where $\sum_{i=0}^d\delta_i=m$ is not prime. Then one has 
$$i_k+i_{\ell} \geq i_{k+\ell} \;\text{ for }\; 1 \leq k \leq \ell \leq g-1 
\;\text{ with }\; k+\ell \leq g-1, $$ 
where $g$ is the least prime divisor of $m$. 
\end{Proposition}
\begin{proof}
By applying \cite[Theorem 13]{Karo}, 
a proof of this statement is by with the same as the proof of Theorem \ref{main} (b). 
\end{proof}
In fact, it is immediate that the above example satisfies 
$i_k+i_{\ell} \geq i_{k+\ell}$ for $1 \leq k \leq \ell \leq g-1$ with $k+\ell \leq g-1$.

\end{document}